\documentclass[12pt,a4paper]{article}
\usepackage[top=2.5cm,bottom=2.5cm,left=2.5cm,right=2.5cm]{geometry}

\usepackage{amssymb,amsmath,graphicx,color,amsthm,centernot}
\usepackage[capitalise]{cleveref}
\usepackage[labelformat=simple]{subcaption}
\usepackage[font=small,labelfont=bf,width=12.5cm]{caption}

\usepackage{multirow,tabularx}

\newtheorem{claim}{Claim}
\newtheorem{theorem}{Theorem}[section]
\newtheorem{lemma}[theorem]{Lemma}

\newtheorem{conjecture}[theorem]{Conjecture}

\newcommand{\set}[1]{\ensuremath{\left\{#1 \right\}}}

\newcommand{\chin}[1]{\ensuremath{\chi_{n}'(#1)}}

\newenvironment{proofclaim}[1][]%
    {\noindent \emph{Proof.} {}{#1}{}}{$~$\hfill $~\blacklozenge$ \vspace{0.2cm}}

\definecolor{defblue}{rgb}{0.4,0,0.84}
\definecolor{greyblue}{rgb}{0.23,0.4,0.70}
\definecolor{orange}{rgb}{1.0,0.5,0.2}
\definecolor{violet}{rgb}{0.55,0,0.55}

\usepackage{url}
\makeatletter
\g@addto@macro{\UrlBreaks}{\UrlOrds}
\makeatother

\newcolumntype{Y}{>{\centering\arraybackslash}X}

\begin{document}

\title{{\bf Normal $6$-edge-colorings of cubic graphs with oddness $2$}}

\author
{
	Igor Fabrici\thanks{Pavol Jozef \v Saf\'{a}rik University, Faculty of Science, Ko\v{s}ice, Slovakia} \quad
	Borut Lu\v{z}ar\thanks{Faculty of Information Studies in Novo mesto, Slovenia.} \thanks{Rudolfovo - Science and Technology Centre Novo mesto, Slovenia.} \quad
	Roman Sot\'{a}k\footnotemark[1] \quad
	Diana \v{S}vecov\'{a}\footnotemark[1]
}

\maketitle

{
\begin{abstract}
	A {\em normal edge-coloring} of a cubic graph is a proper edge-coloring, 
	in which every edge is adjacent to edges colored with four distinct colors
	or to edges colored with two distinct colors.
	It is conjectured that $5$ colors suffice for a normal edge-coloring of any bridgeless cubic graph
	and this statement is equivalent to the Petersen Coloring Conjecture.	
	
	In this paper, we extend the result of Mazzuoccolo and Mkrtchyan 
	({\em Normal 6-edge-colorings of some bridgeless cubic graphs, Discrete Appl. Math. 277 (2020), 252--262}),
	who proved that every cycle permutation graph admits a normal edge-coloring with at most $6$ colors.
	In particular, we show that every cubic graph with oddness $2$ admits a normal edge-coloring with at most $6$ colors.
\end{abstract}
}


\section{Introduction}

A {\em normal edge-coloring} of a cubic (multi)graph is a proper edge-coloring, 
in which every edge is adjacent to edges colored with four distinct colors 
(we call such edges {\em rich})
or to edges colored with two distinct colors
(we call such edges {\em poor}).
If for a normal edge-coloring at most $k$ colors are used, we refer to it as a {\em normal $k$-edge-coloring}.
The smallest $k$, for which a graph $G$ admits a normal $k$-edge-coloring is the {\em normal chromatic index} of $G$,
denoted by $\chin{G}$.
If a cubic graph admits a proper edge-coloring with $3$ colors ({\em class I graphs}), 
then every edge is poor, and hence the coloring is also normal.
Therefore, the only interesting classes of graphs for a normal coloring are cubic graphs with chromatic index $4$ ({\em class II graphs}),
particularly the class of snarks.

The normal edge-coloring was defined by Jaeger~\cite{Jae85}, 
who proposed the following conjecture.
\begin{conjecture}[Jaeger~\cite{Jae85}]
	\label{con:Nor}
	For any bridgeless cubic graph $G$, it holds that
	$$
		\chin{G} \le 5\,.
	$$
\end{conjecture}
Jaeger showed that a normal $5$-edge-coloring of a bridgeless cubic graph $G$ is equivalent 
to a Petersen coloring of $G$, and hence that
Conjecture~\ref{con:Nor} is equivalent to the Petersen Coloring Conjecture~\cite{Jae88},
which asserts that the edges of every bridgeless cubic graph $G$ can be colored by using the edges of the Petersen graph $P$ as colors
in such a way that adjacent edges of $G$ are colored by adjacent edges of $P$;
in particular, a bridgeless cubic graph admits a normal $5$-edge-coloring if and only if it admits a Petersen coloring.

In general, it is known that every cubic graph (with the bridgeless condition omitted) admits a normal $7$-edge-coloring~\cite{MazMkr20}, 
and the bound is tight, e.g., by any cubic graph that contains as a subgraph the complete graph $K_4$ with one edge subdivided.
When considering only bridgeless cubic graphs, 
already proving that a graph is normally $6$-edge-colorable is challenging,
and only a handful of results were published.
For example, Mazzuoccolo and Mkrtchyan~\cite{MazMkr20b} proved that every claw-free cubic graph, tree-like snark, 
and permutation snark (i.e., cycle permutation cubic graphs of class II) admit a normal $6$-edge-coloring.
With at most $5$ colors available, only very particular graphs are known to admit a normal edge-coloring, 
see, e.g.,~\cite{FerMazMkr20,HagSte13,SedSkr24b,SedSkr24,ZhoHaoLuoLuo26}.
Hence, Conjecture~\ref{con:Nor} remains widely open in general.

In this paper, we improve the result of Mazzuoccolo and Mkrtchyan~\cite{MazMkr20b} on permutation snarks;
in particular, we extend it to bridgeless cubic graphs with oddness $2$, 
i.e., bridgeless cubic graphs of class II which admit a $2$-factor with exactly two odd cycles.
\begin{theorem}
	\label{thm:main}
	Every bridgeless cubic graph with oddness at most $2$ admits a normal $6$-edge-coloring.
\end{theorem}

In our proof, we use the idea of Mazzuoccolo and Mkrtchyan used for proving that permutation snarks admit normal $6$-edge-coloring.
They assign the nonzero elements of the elementary abelian group $\mathbb{Z}_2^3$ to the edges of a graph
in order to construct a nowhere-zero $\mathbb{Z}_2^3$-flow, and finally replace the elements of the group
by $6$ colors, showing that two elements can be represented by the same color to obtain a normal $6$-edge-coloring.
We extend this approach to graphs with oddness $2$.

\section{Preliminaries}
\label{sec:prel}

In this section, we introduce terminology, notation, 
and present several auxiliary results needed in our proof.

A vertex of degree $k$ is called a {\em $k$-vertex}.
For a set of vertices or edges $X$, 
by $G \setminus X$ we denote the graph with the objects from $X$ removed;
similarly, we denote by $G \cup X$ the graph with the objects from $X$ added.

We refer to an edge which is either poor or rich as {\em normal}.
We call an edge-coloring of a graph, in which some edges might not be colored,
a {\em partial normal edge-coloring} if no two adjacent edges have the same color,
and every edge which has all adjacent edges colored is normal.

By a classical result of Petersen~\cite{Pet1891}, we know that
every bridgeless cubic graph admits a perfect matching and 
therefore also a $2$-factor.
The structure of $2$-factors in cubic graphs plays a special role in various types 
of edge-colorings.
{\em Oddness} of a graph $G$, denoted by $\omega(G)$, is the minimum number of odd cycles in a $2$-factor of $G$
over all $2$-factors of $G$.
Let us remark here that in a proper edge-coloring of a bridgeless cubic graph with oddness $2$,
the number of edges of the color appearing the least number of times 
(usually, we will use the coloring with colors from $\set{1,2,3,4}$ 
and the color class $4$ will be the smallest) is at least $2$.

Given an edge-coloring $\varphi$, we denote the set of colors appearing on the edges incident to a vertex $v$
with $C_{\varphi}(v)$ or simply $C(v)$ if the coloring is clear from the context.
Given a proper edge-coloring of a graph $G$, the maximal path starting at vertex $x$ and having all edges colored with colors $\alpha$ and $\beta$
is called the {\em $(x;\alpha,\beta)$-path}.
When the starting vertex is irrelevant, we refer to the path as to an {\em $(\alpha,\beta)$-path}.
We say that we {\em swap the colors} along an $(x;\alpha,\beta)$-path $P$ if we recolor every edge of color $\alpha$ of $P$ to a color $\beta$
and vice-versa.

The following lemmas discuss operations on bridgeless cubic graphs with oddness $2$, 
which do not increase 
the oddness in the modified graph.

Given a $3$-cycle $T$ in a cubic graph $G$,
the {\em $\Delta$-reduction of $G$ by $T$} or simply a {\em $\Delta$-reduction of $G$} 
is the operation identifying the vertices of $T$ into one vertex,
i.e., the three edges of $T$ are contracted.
\begin{lemma}[Triangle Reduction Lemma]
	\label{lem:3reduce}
	Let $G$ be a bridgeless cubic graph with $\omega(G)=2$ and at least one $3$-cycle $T$.
	Then, the graph $G'$ obtained from $G$ after a $\Delta$-reduction by $T$ is
	also a bridgeless cubic graph with $\omega(G')=2$.
\end{lemma}

\begin{proof}	
	Let $T = uvw$ be a $3$-cycle in $G$ such that $G'$ is obtained after a $\Delta$-reduction of $G$ by $T$.

	We first show that $G'$ is bridgeless.
	Suppose, to the contrary, that $G'$ contains a bridge $xy$. 
	If $T$ reduces into a vertex different from $x$ or $y$, then the bridge $xy$ 
	must also be present in $G$, a contradiction. 
	We can therefore assume that $T$ reduces into one of the vertices $x$ and $y$, say $x$.
	Then, in $G$, one of the vertices $u, v, w$, say $u$, is incident with $y$, but then $uy$ is a bridge in $G$, a contradiction. 

	Now we show that $\omega(G') = 2$.
	Let $F$ be a $2$-factor of $G$ with two odd cycles.
	Since $\omega(G) = 2$, $G$ is not properly $3$-edge-colorable,
	but there is a proper edge-coloring $\varphi$ of $G$
	such that the edges of the matching $M = E(G)\setminus E(F)$ are colored with $1$,
	the edges of the even cycles of $F$ with colors $2$ and $3$,
	one edge of each of the two odd cycles of $F$ with color $4$, 
	and the rest of their edges with colors $2$ and $3$.	
	We will show that $G'$ also admits a proper edge-coloring $\varphi'$ with exactly two edges of color $4$.
	This will imply that $\omega(G') = 2$, otherwise $\omega(G') = 0$ meaning that $G'$ would admit a proper $3$-edge-coloring $\varphi''$.
	However, in such a case, the three edges incident with $x$ in $G'$ would be colored with three distinct colors,
	and $\varphi''$ would induce a partial edge-coloring in $G$ with only the edges of $T$ non-colored.
	It is easy to see that the coloring can be extended to the three edges of $T$, a contradiction.
	Hence, if $\omega(G') \le 2$, then $\omega(G') = 2$.
	
	In order to prove that $\omega(G') \le 2$, we consider two cases.
	
	\medskip
	\noindent {\bf Case 1:} \quad
	Suppose that $T$ is in $F$ and, without loss of generality, assume
	that $\varphi(uv) = 2$, $\varphi(vw) = 3$, and $\varphi(uw) = 4$.
	Clearly, all the edges with exactly one endvertex incident with $T$, denote them $e_u=ux_1$, $e_v=vx_2$, and $e_w=wx_3$ (by their endvertex in $T$) are in $M$.
	Now, consider the coloring $\varphi'$ of $G'$ induced by $\varphi$. 	
	Let $x$ be the vertex corresponding to the identified vertices of $T$ in $G'$.
	The three edges incident to $x$ (i.e., the edges $e_u$, $e_v$, and $e_w$) 
	are all colored by $1$ and $x$ is the only
	vertex at which incident edges have the same color.	
	Now, uncolor the three edges incident with $x$.
	
	Observe that there is exactly one edge of color $4$ in $G'$ (the one on the odd cycle of $F$ distinct from $T$).	
	Therefore, at least two of the vertices in $\set{x_1,x_2,x_3}$ are incident with an edge of color $2$.
	Moreover, the graph induced on the edges of colors $1$ or $2$ is a union of cycles and two paths,
	whose endvertices are $x_1$, $x_2$, $x_3$, and a vertex incident with the edge of color $4$
	(note that the latter vertex may be one of $x_1$, $x_2$, $x_3$, in which case one of the two paths has length $0$).
	Without loss of generality, we may assume that the endvertices of a path $P$ of length more than $0$ are $x_1$ and $x_2$ (see Figure~\ref{fig:delta-red-c1}).
	\begin{figure}[htp!]
		$$
			\includegraphics{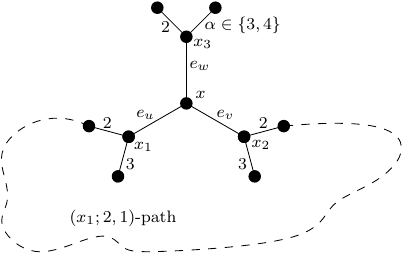}
		$$
		\caption{When $T$ is in $F$, after uncoloring the edges incident to $x$,
			there is a $(1,2)$-path in $G'$ starting and ending in two vertices adjacent to $x$,
			without loss of generality, we can assume those are the vertices $x_1$ and $x_2$.}
		\label{fig:delta-red-c1}
	\end{figure}
	
	Now, swap the colors on the path $P$ and observe that $x_1$ and $x_2$ are now both incident with
	edges of colors $1$ and $3$ (or $4$ if they are adjacent).
	Therefore, we can color $e_u$ with $2$, $e_v$ with $4$ (or $3$ if $x_1$ and $x_2$ are adjacent and $\varphi'(x_1x_2) = 4$), and $e_w$ with $1$.	
	In this way, $\varphi'$ is a proper coloring of $G'$ using exactly two edges of color $4$.
	This means that $\omega(G') \le 2$.
	
	\medskip
	\noindent {\bf Case 2:} \quad	
	Suppose that $T$ is not in $F$;
	in particular this means that there is an edge of $T$, say $uv$, in $M$.
	Let $C$ be the cycle of $F$ containing $vw$ and $uw$.
	Clearly, it contains also the two edges $e_u$ and $e_v$ incident with $u$ and $v$ not on $T$, respectively.
	Now, if one of the edges $vw$ and $uw$ is colored with $4$, say $uw$, then we swap the colors of $e_u$ and $uw$.
	Note that after the $\Delta$-reduction of $G$ by $T$, the parity of the length of $C'$, 
	the cycle in $G'$ corresponding to $C$ in $G$ after contracting the edges $vw$ and $uw$, 
	remains the same as the parity of the length of $C$, 
	and so $G'$ has a $2$-factor with exactly two odd cycles.
	Moreover, the coloring $\varphi'$ induced by $\varphi$ is a proper edge-coloring using exactly two edges of color $4$,
	implying that $\omega(G') \le 2$.
	This completes the proof.
\end{proof}

A similar statement to the following, presented without proof, appears in Lukot{\!}'ka~et al.~\cite[Lemma~4]{LukMacMazSko15}.

\begin{lemma}[$3$-Edge-Cut Reduction Lemma]
	\label{lem:3cut-reduce}
	Let $G$ be a $3$-edge-connected cubic graph with $\omega(G)=2$ and a non-trivial $3$-edge-cut $C = \set{u_1v_1,u_2v_2,u_3v_3}$.
	Let $G_u$ be the component of $G - C$ containing the vertices $u_1$, $u_2$, and $u_3$,
	together with an additional vertex $x_u$ and the edges $u_1x_u$, $u_2x_u$, and $u_3x_u$.
	Similarly, let $G_v$ be the component of $G - C$ containing the vertices $v_1$, $v_2$, and $v_3$,
	together with an additional vertex $x_v$ and the edges $v_1x_v$, $v_2x_v$, and $v_3x_v$.
	Then, $G_u$ and $G_v$ are also bridgeless cubic graphs with oddness at most $2$.
\end{lemma}

\begin{proof}
	First, note that if any pair of the edges of $C$ have a common endvertex, then there is a 2-edge-cut in $G$, contradicting 3-edge-connectivity of $G$.
	Thus, the vertices $u_1,\dots,v_3$ are pairwise distinct.

	Let $G$, $G_u$, and $G_v$ be graphs as defined in the lemma.
	Let $F$ be a 2-factor of $G$ with two odd cycles and $M = G \setminus E(F)$ the corresponding perfect matching.	
	Clearly, either one or three edges of $C$ are included in $M$,
	and in the former case, the other two edges belong to the same cycle of $F$.

	\medskip
	\noindent {\bf Case 1:} \quad
	Suppose first that all three edges of $C$ are in $M$.
	Hence, the two odd cycles of $F$ either both appear in one of $G_u$ and $G_v$, or one appears in $G_u$ and the other in $G_v$.
	We consider the two subcases separately.

	Suppose that both odd cycles are in one of $G_u$ and $G_v$;
	without loss of generality, 
	we may assume that $G_u$ contains no odd cycle of $F$ and $G_v$ contains both of them.
	So, all vertices of $G_u$ except $x_u$ lie on even cycles, and thus $G_u$ has an odd number of vertices,
	which is not possible since $G$ is cubic.
	Therefore, both $G_u$ and $G_v$ contain exactly one odd cycle of $F$.

	Next, we color the edges of $G$ so that the edges of $M$ receive color $1$, while the edges of $F$ receive colors $2$, $3$, and $4$.
	Moreover, we use color $4$ at most two times---specifically, on one edge in each of the two odd cycles of $F$. We denote this coloring by $\varphi$.
	Hence, in the coloring $\varphi_u$ of $G_u$ induced by $\varphi$, the edges $u_1x_u$, $u_2x_u$, and $u_3x_u$ receive color $1$ 
	(note that $\varphi_u$ is not proper yet).
	We now use $\varphi_u$ to construct a $2$-factor $F_u$ of $G_u$ that contains two odd cycles.

	First, suppose that none of the vertices $u_1, u_2, u_3$ is incident with an edge of color $4$.
	That is, all of them are incident only with edges of colors $1$, $2$, and $3$.
	Since there is only one vertex in $G_u$ which is not incident with edges of colors $1$ and $3$, 
	there exists a cycle $K = x_uu_i\dots u_jx_u$ such that $\{i,j\}\subset \{1, 2, 3\}$ and $\varphi_u(e) \in \{1, 3\}$ for every $e\in E(K)$.
	We swap the colors $1$ and $3$ along the edges of $K$.
	Now, two edges incident with $x_u$ have color $3$ and one of them has color $1$.
	We recolor one of the edges colored with $3$ to $4$.
	The coloring $\varphi_u$ modified like this is proper, and every vertex of $G_u$ is incident with an edge of color $1$, 
	which means that the edges of color $1$ form a perfect matching in $G_u$.
	Moreover, all the other edges of $G_u$ are colored with $2$, $3$, or $4$, while only two edges have color $4$.
	Since $G_u$ is cubic, it follows that it has at most two odd cycles.

	Second, we may assume that, without loss of generality, $u_1$ is incident with edges of colors $1$, $3$, and $4$.
	Moreover, at least one of the vertices $u_2$ and $u_3$, say $u_2$, is not incident with an edge of color $4$.
	We change the color of the edge $x_uu_1$  from $1$ to $2$, and the color of the edge $x_uu_2$  from $1$ to $4$.
	The modified coloring $\varphi_u$ is proper, and every vertex of $G_u$ is incident with an edge of color $2$. 
	The edges with color $2$ now form a perfect matching $M_u$ of $G_u$.
	Furthermore, only two edges of $\varphi_u$ received color $4$, which implies that there exists a 2-factor $F_u$ of $G_u$ 
	that contains at most two odd cycles, consisting of edges with colors $1$, $3$, and $4$.
	By an analogous argument, the same holds for $G_v$.

	\medskip
	\noindent {\bf Case 2:} \quad	
	Suppose that only one edge of $C$ is in $M$.
	Note that the two edges of $C$ belonging to $F$ are in the same cycle, we denote this cycle as $D$.
	Also, we denote as $D_u$ the cycle of $G_u$ corresponding to the cycle $D$ in $G$ (similarly $D_v$).

	If both $G_u$ and $G_v$ contain at most two odd cycles in 2-factors corresponding to $F$ in $G$, then we are done.
	So suppose that, without loss of generality, $G_u$ contains three odd cycles in the $2$-factor $F_u$ corresponding to $F$ 
	(it is obvious that more than three cannot arise).
	This means that $G_u$ has a $2$-factor with an odd number of odd cycles, a contradiction.
\end{proof}

\section{Proof of Theorem~\ref{thm:main}}
\label{sec:main}

In this section, we prove Theorem~\ref{thm:main}.

\begin{proof}[Proof of Theorem~\ref{thm:main}]	
	Let $G$ be a minimal counterexample to the theorem; 
	i.e., a bridgeless cubic graph with oddness $2$ and the minimum number of vertices
	such that it does not admit a normal $6$-edge-coloring.	
	Hence $G$ is not isomorphic to the Petersen graph,
	which admits a normal $5$-edge-coloring.
	Moreover, by the minimality, $G$ is connected.
	
	We first determine several properties of $G$.	
	\begin{claim}
		\label{cl:multi}
		$G$ does not contain any multiedge.
	\end{claim}
	
	\begin{proofclaim}
		First, consider the trivial case when
		$G$ is a tripole, i.e., a graph with two vertices connected with three edges.
		Then $G$ is in class I and thus normal $3$-edge-colorable.
		
		Suppose now that there is a pair of vertices $u$ and $v$ connected with two edges.
		Let $u'$ and $v'$ be the other neighbor of $u$ and $v$, respectively.
		Note that $u'\neq v'$, otherwise $G$ would have a bridge.
		Let $G'$ be the graph obtained from $G$ by removing $u$ and $v$
		and connecting $u'$ and $v'$ (note that a multiedge may appear).
		
		We first show that $G'$ also has oddness at most $2$.
		Let $F$ be a $2$-factor of $G$ with two odd cycles and $M$ the matching $E(G)\setminus E(F)$.		
		If the digon between $u$ and $v$ in $G$ is a cycle in $F$,
		then $uu'$ and $vv'$ are in $M$ and there is a $2$-factor in $G'$
		comprised of all cycles from $F$ except the digon $uv$.
		Otherwise, one of the digon edges is in $M$ and the other, 
		together with the edges $uu'$ and $vv'$ are consecutive edges of some cycle $C$ in $F$.
		In this case, there is a $2$-factor in $G'$ comprised of all cycles from $F$
		with the cycle $C$ shortened by two edges (the three consecutive edges being replaced by $u'v'$),
		and hence the new cycle has the length of the same parity as $C$. 
		It follows that there is a $2$-factor in $G$ with two odd cycles.
		
		Therefore, by the minimality of $G$,
		there is a normal $6$-edge-coloring $\varphi'$ in $G'$.
		Let $\varphi$ be a partial normal $6$-edge-coloring of $G$ induced by $\varphi'$
		with the edges incident to $u$ and $v$ being noncolored.
		Then, we can extend the coloring to all the edges of $G$
		by setting $\varphi(uu') = \varphi(vv') = \varphi'(u'v')$
		and coloring the two edges incident with $u$ and $v$ 
		with the two colors $C(u')\setminus \set{\varphi'(u'v')}$.
	\end{proofclaim}

	\begin{claim}
		\label{cl:2cut}
		$G$ does not contain any $2$-edge-cut.
	\end{claim}
	
	\begin{proofclaim}
		Suppose the contrary and let $\set{e_1,e_2}$ be a $2$-edge-cut in $G$.
		Clearly, $e_1$ and $e_2$ are not incident, otherwise $G$ would have a bridge.
				
		Next, let $G_1$ and $G_2$ be the components of $G \setminus \set{e_1,e_2}$ 
		and denote by $u_1$ and $u_2$ ($v_1$ and $v_2$) the $2$-vertices in $G_1$ ($G_2$) 
		incident with $e_1$ and $e_2$, respectively.
		Let $G_1' = G_1 \cup \set{u_1u_2}$ and $G_2' = G_2 \cup \set{v_1v_2}$.
		Note that $G_1'$ and $G_2'$ are bridgeless cubic graphs. 
		We will first show that they also have oddness at most $2$.
		
		Let $F$ be a $2$-factor of $G$ with two odd cycles and $M$ the corresponding matching.
		Observe that the edges $e_1$ and $e_2$ are either the matching edges or they appear on the same cycle in $F$.
		In the former case, $F$ induces a $2$-factor in $G_1'$ and $G_2'$,
		and thus each of them has oddness at most $2$.
		Therefore, we may assume that $e_1$ and $e_2$ appear on a cycle $C$ of $F$.
		Again, observe that  $F$ together with the new edges $u_1u_2$ and $v_1v_2$ induces a $2$-factor in $G_1'$ and $G_2'$, respectively.		
		Moreover, the cycle $C$ splits into two cycles, one, call it $C_1$, in $G_1'$ and the other, call it $C_2$, in $G_2'$.
		Since the $2$-factor in $G_1'$ without $C_1$ contains at most two odd cycles, together with $C_1$,
		it contains at most three.
		Since the number of odd cycles in any $2$-factor of a cubic graph is even, it means that $G_1'$
		contains at most two and hence $\omega(G_1') \le 2$. 
		Analogously, we deduce that $\omega(G_2') \le 2$.
		
		Therefore, by the minimality of $G$,
		there are normal $6$-edge-colorings $\varphi_1'$ and $\varphi_2'$ of $G_1'$ and $G_2'$, respectively.
		Now, we permute the colors in $\varphi_2'$ such that the color of $v_1v_2$ is the same as the color of $u_1u_2$,
		and that the other two colors at $v_1$ match the other two colors at $u_1$.
		Finally, if both edges $u_1u_2$ and $v_1v_2$ are rich, then we also permute the colors so that
		the colors at $u_2$ and $v_2$ match.		
		The colorings $\varphi_1'$ and $\varphi_2'$ induce a partial $6$-edge-coloring $\varphi$ of $G$,
		and we complete it with assigning the color $\varphi_1'(u_1u_2)$ to the edges $e_1$ and $e_2$,
		a contradiction.
	\end{proofclaim}

	\begin{claim}
		\label{cl:3cyc}
		$G$ does not contain any $3$-cycle.
	\end{claim}	

	\begin{proofclaim}
		Suppose, to the contrary, that $G$ contains a $3$-cycle $T$.
		Let $G'$ be the graph obtained from the $\Delta$-reduction of $G$ by $T$
		and let $x$ be the vertex to which the vertices of $T$ collapse in $G'$.
		By Lemma~\ref{lem:3reduce}, $G'$ has oddness $2$ and by the minimality,
		it admits a normal $6$-edge-coloring $\varphi'$.
		Let $\varphi$ be the partial $6$-edge-coloring of $G$ induced by $\varphi'$.
		In order to complete $\varphi$ to a normal $6$-edge-coloring of $G$, 
		we only need to color the edges of $T$.
		Note that $|C(x)| = 3$ in $\varphi'$ and therefore the three edges incident to $T$
		and not on $T$ have distinct colors, say $1$, $2$, and $3$.
		Moreover, in $\varphi'$, each of the three edges was adjacent to the other two,
		and so we can color also the edges of $T$ with the colors $1$, $2$, and $3$,
		fulfilling the conditions of the normal coloring.
		In this way, $\varphi$ becomes a normal coloring of $G$, a contradiction.
	\end{proofclaim}

	So far, we know that $G$ is a connected, simple, $3$-edge-connected, triangle-free cubic graph with oddness $2$.
	We now proceed to establish a property concerning special $3$-edge-cuts.

	\begin{claim}
		\label{cl:special3cut}
		In $G$, there is no $3$-edge-cut $C = \set{u_1v_1,u_2v_2,u_3v_3}$
		such that $v_1v_2 \in E(G)$ and that there exists a vertex $v_4$ 
		such that $v_2v_4,v_3v_4 \in E(G)$.
	\end{claim}	
	
	\begin{proofclaim}	
		Suppose the contrary and let $C$ be a $3$-edge-cut in $G$ 
		with the properties as described in the claim.
		We assume a labeling of relevant vertices as given in Figure~\ref{fig:special3cut}.
		\begin{figure}[htp!]
			$$
				\includegraphics{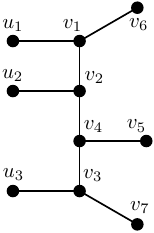}
			$$
			\caption{A $3$-edge-cut from Claim~\ref{cl:special3cut}.}
			\label{fig:special3cut}
		\end{figure}
				
		Let $G_L$ be the component of $G \setminus C$ containing $u_1$, $u_2$, and $u_3$,
		and with an additional vertex $x$ connected to $u_1$, $u_2$, and $u_3$.
		Hence, $G_L$ is cubic and, by Lemma~\ref{lem:3cut-reduce}, $G_L$ has oddness at most $2$.
		Therefore, by the minimality of $G$, 
		$G_L$ admits a normal $6$-edge-coloring $\varphi_L$;
		say $\varphi_L(u_1x) = 1$, $\varphi_L(u_2x) = 2$, and $\varphi_L(u_3x) = 3$.
		
		Now, note that $C' = \set{v_1v_6, v_4v_5, v_3v_7}$ is also a $3$-edge-cut of $G$.
		Let $G_R$ be the component of $G \setminus C'$ containing $v_5$, $v_6$, and $v_7$,
		and with an additional vertex $y$ connected to $v_5$, $v_6$, and $v_7$.
		Hence, also $G_R$ is cubic and, by Lemma~\ref{lem:3cut-reduce}, it has oddness at most $2$.
		Again, by the minimality of $G$, 
		$G_R$ admits a normal $6$-edge-coloring $\varphi_R$;
		without loss of generality, 
		we may assume that $\varphi_R(v_6y) = 2$, $\varphi_R(v_5y) = 3$, and $\varphi_R(v_7y) = 1$.
		
		Finally, we use the colorings $\varphi_L$ and $\varphi_R$ to induce 
		a partial normal $6$-edge-coloring $\varphi$ of $G$. 
		We complete the coloring $\varphi$ by setting
		$\varphi(u_1v_1) = 1$, $\varphi(u_2v_2) = 2$, $\varphi(u_3v_3) = 3$,
		$\varphi_R(v_6v_1) = 2$, $\varphi_R(v_5v_4) = 3$, $\varphi_R(v_7v_3) = 1$,
		$\varphi_R(v_1v_2) = 3$, $\varphi_R(v_2v_4) = 1$, and $\varphi_R(v_3v_4) = 2$ (see Figure~\ref{fig:special3cut2}).
		\begin{figure}[htp!]
			$$
				\includegraphics{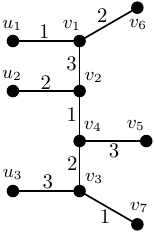}
			$$
			\caption{Coloring of the non-colored edges after inducing $\varphi$ from $\varphi_L$ and $\varphi_R$.}
			\label{fig:special3cut2}
		\end{figure}		
		It is easy to see that $\varphi$ is indeed a normal $6$-edge-coloring of $G$,
		a contradiction.
	\end{proofclaim}

	\bigskip
	Having established the above properties of $G$, we proceed with the second step.
	We construct a specific normal $6$-edge-coloring of $G$, and thus reach a contradiction.
	
	For a path $P = v_1\dots v_k$, we denote its subpath from a vertex $v_i$ to a vertex $v_j$,
	where $1 \le i \le j \le k$, by $[v_i,v_j]_P$, and refer to it as a {\em segment of $P$}.
	Sometimes, we also refer to a path $P$ by specifying only its first vertex $u$ and last vertex $v$;
	in such cases, we denote it by $[u,v]$.
	
	Given two disjoint cycles $C_1$ and $C_2$, we say that a path $P = [x,y]$ is a {\em $(C_1,C_2)$-path} 
	if $x$ is a vertex of $C_1$ (we call $x$ the {\em starting vertex}), 
	$y$ is a vertex of $C_2$ (we call $y$ the {\em terminal vertex}),
	and $P$ contains no edge of $C_1$ or $C_2$.	
	A pair of disjoint $(C_1,C_2)$-paths $P_1$ and $P_2$ is called {\em good} with respect to a subgraph $S$
	if no vertex $v_1 \in V(P_1)$ is adjacent to any vertex $v_2 \in V(P_2)$ via an edge of $S$.
	When $S$ is clear from the context, we simply refer to the two paths as {\em good}.

	Given a cycle $C$ and two vertices $x,y \in V(C)$, 
	we refer to each of the two subpaths of $C$ between $x$ and $y$ as an {\em $[x,y]_C$-segment}.
	
	Let $F$ be a $2$-factor of $G$ containing exactly two odd cycles, $C_1$ and $C_2$.
	Our next goal is to find two good $(C_1,C_2)$-paths with respect to $F$, 
	based on which we will construct a normal $6$-edge-coloring of $G$.	
	To identify such paths, we first prove the following claim,
	which guarantees their existence in any graph satisfying the properties 
	established for $G$ in Claims~\ref{cl:multi}-\ref{cl:special3cut}.
	
	\begin{claim}
		\label{cl:goodpaths}
		Let $H$ be a simple, $3$-edge-connected, triangle-free cubic graph with oddness $2$,
		containing no $3$-edge-cuts as described in Claim~\ref{cl:special3cut}.
		Let $F$ be a $2$-factor of $H$ with exactly two odd cycles, $C_1$ and $C_2$.
		Then, either $H$ contains a pair of good $(C_1,C_2)$-paths with respect to $F$, or $H$ is isomorphic to the Petersen graph.
	\end{claim}
	
	\begin{proofclaim}
		We prove the claim by contradiction.
		Assume that $H$ is a smallest graph (with respect to the number of vertices) 
		satisfying the given properties for which the claim fails to hold.
		
		Let $F$ be a $2$-factor of $H$ with exactly two odd cycles $C_1$ and $C_2$.	
		We denote the matching $H \setminus E(F)$ by $M$.
		Note that by Claim~\ref{cl:3cyc}, both odd cycles have length at least~$5$.
		
		Next, let $H^*$ be the graph obtained from $H$ by contracting the edges of the even cycles in $F$,
		and then simplifying the resulting graph by deleting any loops and retaining only a single edge between each pair of adjacent vertices.		
		Note that in $H^*$, the vertices of $C_1$ and $C_2$ retain degree $3$, while the degree of any other vertex may be any positive integer.
		Clearly, a good pair of $(C_1,C_2)$-paths with respect to $C_1 \cup C_2$ in $H^*$ implies existence of a good pair of $(C_1,C_2)$-paths with respect to $F$ in $H$.
		
		We consider two cases regarding the number of disjoint $(C_1,C_2)$-paths in $H^*$.
		
		\medskip
		\noindent {\bf Case 1:} \quad
		{\em There are at least two disjoint $(C_1,C_2)$-paths in $H^*$.}

		Let $P_1^* = [x_1,y_1]$ and $P_2^* = [x_2,y_2]$ be such two paths.
		If $P_1^*$ and $P_2^*$ are good, then we are done.
		Otherwise, their starting vertices or terminal vertices are adjacent.	

		Without loss of generality, we may assume that $x_1$ and $x_2$ are adjacent.
		Since $H$ is $3$-edge-connected, 
		there exists a $(C_1,C_2)$-path $P_3^* = [x_3,y_3]$ with $x_3 \notin \set{x_1,x_2}$.
		Moreover, since $C_1$ has length at least $5$ and $H$ contains no $2$-edge-cut,
		we can choose $P_3^*$ such that $x_3$ is adjacent to neither $x_1$ nor $x_2$
		(note that $P_3^*$ is not necessarily disjoint from $P_1^*$ or $P_2^*$).
		
		We now consider two cases, and in each, we show that either there exist two good $(C_1,C_2)$-paths in $H$, 
		or $H$ is isomorphic to the Petersen graph.

		\medskip
		\noindent {\bf Case 1.1:} \quad	
		{\em Suppose that $P_3^*$ contains a vertex that also lies on $P_1^*$ or $P_2^*$.}

		Let $w$ be the first such vertex on $P_3^*$ starting from $x_3$, say $w \in V(P_1^*)$.
		Consider the path $P_3^\dag = [x_3,w]_{P_3^*} \cup [w,y_1]_{P_1^*}$ (note that $w \ne y_1$, since $H$ is cubic).
		If $y_1$ and $y_2$ are not adjacent, then $P_3^\dag$ and $P_2^*$ form a good pair.
		Hence, we may assume that $y_1$ and $y_2$ are adjacent.

		Now, similarly as above, there is a $(C_1,C_2)$-path $P_4^* = [x_4,y_4]$ with $y_4 \notin \set{y_1,y_2}$
		such that $y_4$ is not adjacent neither to $y_1$ nor $y_2$.
		If $P_4^*$ has no common vertex with $P_1^*$, $P_2^*$, and $P_3^\dag$, 
		then $P_4^*$ and $P_i^*$ are good, where $i \in \set{1,2}$ is such that $x_4$ and $x_i$ are not adjacent
		(recall that $x_1$ and $x_2$ are adjacent and thus such an $i$ exists).
		Otherwise, $P_4^*$ has a common vertex with $P_1^*$ or $P_2^*$, or $P_3^\dag$;
		let $z$ be the first such vertex starting from $y_4$. 
		\begin{itemize}
			\item{} If $z \in V(P_1^*)$, then $P_2^*$ and $[x_3,w]_{P_3^\dag} \cup [w,z]_{P_1^*} \cup [z,y_4]_{P_4^*}$ form a good pair
				(this holds regardless of whether $w$ appears before or after $z$ on $P_1^*$).
			\item{} If $z \in V(P_2^*)$, then $P_3^\dag$ and $[x_2,z]_{P_2^*} \cup [z,y_4]_{P_4^*}$ form a good pair.
			\item{} If $z \in V(P_3^\dag)$, then $P_1^*$ and $[x_3,z]_{P_3^\dag} \cup [z,y_4]_{P_4^*}$ form a good pair.
		\end{itemize}

		\medskip
		\noindent {\bf Case 1.2:} \quad	
		{\em Suppose that $P_3^*$ is disjoint from $P_1^*$ and $P_2^*$.}

		If $y_3$ is not adjacent to both $y_1$ and $y_2$, 
		say, $y_1$ and $y_3$ are not adjacent, then $P_1^*$ and $P_3^*$ form a good pair.
		So, we may assume that $y_3$ is adjacent to both $y_1$ and $y_2$.

		But then, by Claim~\ref{cl:2cut}, there exists a $(C_1,C_2)$-path $P_4^* = [x_4,y_4]$
		such that $y_4 \notin \set{y_1,y_2,y_3}$.
		Since $C_2$ has length at least $5$, $y_4$ is not adjacent to both $y_1$ and $y_2$;
		say that $y_4$ is not adjacent to $y_1$.

		Suppose first that $y_4$ is also not adjacent to $y_2$.
		If $P_4^*$ is disjoint from $P_1^*$, $P_2^*$, and $P_3^*$,
		then there is $i \in \set{1,2,3}$ such that
		$x_4$ is not adjacent to $x_i$ and thus $P_i^*$ and $P_4^*$ form a good pair.
		Otherwise, $P_4^*$ has a common vertex with $P_1^*$, $P_2^*$, or $P_3^*$;
		let $z$ be the first such vertex from $y_4$.
		If $z \in V(P_i^*)$, $i \in \set{1,2}$, then $P_3^*$ and $[x_i,z]_{P_i^*} \cup [z,y_4]_{P_4^*}$ form a good pair.
		Otherwise, $z \in V(P_3^*)$ and a good pair is formed by $P_1^*$ and $[x_3,z]_{P_3^*} \cup [z,y_4]_{P_4^*}$.

		Thus, we may assume that $y_4$ and $y_2$ are adjacent.
		Note that if $P_4^*$ has a common vertex with $P_1^*$, $P_2^*$, or $P_3^*$, 
		then the same paths as described in the above paragraph form good pairs.
		Hence, we may assume that $P_4^*$ is disjoint from $P_1^*$, $P_2^*$, and $P_3^*$.
		Moreover, we may also assume that $x_4$ is adjacent to $x_1$ and $x_3$, 
		otherwise, $P_4^*$ forms a good pair with $P_1^*$ or $P_3^*$, respectively (see Figure~\ref{fig:case1-4paths}).
		\begin{figure}[htp!]
			$$
				\includegraphics{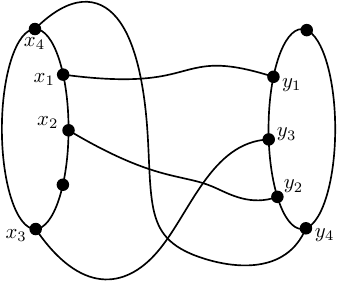}
			$$
			\caption{Four disjoint paths between $C_1$ and $C_2$ in $H^*$.}
			\label{fig:case1-4paths}
		\end{figure}

		By Claim~\ref{cl:2cut} and since $C_2$ has length at least $5$, 
		there is also a $(C_1,C_2)$-path $P_5^* = [x_5,y_5]$, 
		where $y_5 \notin \set{y_1,y_2,y_3,y_4}$.
		If $P_5^*$ has a common vertex with $P_1^*$, $P_2^*$, $P_3^*$, or $P_4^*$,
		then we can again find a good pair of paths in a similar manner as above.
		Moreover, if $P_5^*$ is disjoint from paths $P_1^*$, $P_2^*$, $P_3^*$, and $P_4^*$,
		then we have a good pair unless $y_5$ is adjacent to $y_1$ and $y_4$,
		and $x_5$ is adjacent to $x_2$ and $x_3$.
		This means that $C_1$ and $C_2$ both have length exactly $5$ 
		and each of their vertices is either the starting or the terminal vertex of one 
		of the paths $P_i^*$, $i \in \set{1,2,3,4,5}$.
		
		Now, suppose that there are two paths $P_i^*$ and $P_j^*$ having two internal vertices joined by a path.
		Let $P^* = [w_1,w_2]$ be the shortest such path;
		without loss of generality, we may assume that 
		$w_1 \in V(P_1^*)$ and $w_2 \in V(P_2^*)$.		
		Then, $[x_1,w_1]_{P_1^*} \cup [w_1,w_2] \cup [w_2,y_2]_{P_2^*}$ and $P_5^*$ form a good pair.		
		Hence, we may assume that there are no paths between the five paths (apart from the edges of $C_1$ and $C_2$), 
		and thus the length of each path must be one (i.e., $x_iy_i$ is an edge),
		otherwise $H$ would contain a $2$-edge-cut (comprised of the two edges incident with $x_i$ and $y_i$, not belonging to $C_1$ and $C_2$, respectively), contradicting Claim~\ref{cl:2cut}.
		Therefore, $H$ is isomorphic to the Petersen graph,
		which admits a normal $5$-edge-coloring, a contradiction.

		\medskip
		\noindent {\bf Case 2:} \quad
		{\em There is no pair of disjoint $(C_1,C_2)$-paths in $H^*$.}

		This means that $H^*$ contains a cut-vertex $x$
		such that every $(C_1,C_2)$-path passes through $x$.
		Let $C_x$ be the even cycle of $F$ corresponding to $x$, 
		and let $S_1^*$ and $S_2^*$ be the two components of $H^* \setminus \set{x}$ containing $C_1$ and $C_2$, respectively.
		Define $H_1^* = H^*[S_1^* \cup x]$ and $H_2^* = H^*[S_2^* \cup x]$.	
		Finally, let $H_1$ and $H_2$ be the subgraphs of $H$ corresponding to $H_1^*$ and $H_2^*$, respectively;
		i.e., with the even cycles of $F$ expanded.
		
		Recall that for a good pair of $(C_1,C_2)$-paths $P_1$ and $P_2$ with respect to $F$ in $H$, 
		we require that no vertex of $P_1$ is adjacent to a vertex of $P_2$ via an edge of $F$.
		
		The graphs $H_1$ and $H_2$ both contain $C_x$, along with all of its (eventual) chords.		
		Let $T_1$ and $T_2$ be the sets of edges in $H_1$ and $H_2$, respectively, 
		that have exactly one endvertex from $C_x$.
		By Claim~\ref{cl:2cut}, we have that $|T_1| \ge 3$ and $|T_2| \ge 3$.
		
		Observe that the graph $H_1$ contains exactly $|T_2|$ vertices of degree $2$, which we denote by $V_{x,1}$.
		Moreover, $|T_2|$ must be odd; otherwise, we could connect pairs of vertices from $V_{x,1}$
		and obtain a cubic graph with a $2$-factor containing exactly one odd cycle, which is not possible.
		A symmetric argument shows that $|T_1|$ is also odd. 
		
		We now describe two constructions of graphs obtained from $H_1$ and $H_2$, 
		with the properties required by Claim~\ref{cl:goodpaths},
		which will be used in the final step of this case.
		We consider two cases regarding the size of $T_r$, for $r \in \set{1,2}$.

		\medskip
		\noindent {\bf Construction A} (for $|T_r| = 3$):
		
		Label the vertices of $T_r$ consecutively as they appear on $C_x$ by $x_1$, $x_2$, and $x_3$,
		and denote the edge with $x_i$ as its only endvertex on $C_x$ by $e_i$, for every $i \in \set{1,2,3}$.
		
		Let $H_r^{i}$ be the graph obtained from $H_r$ by first removing all chords of $C_x$ (if any exist),
		and then suppressing all $2$-vertices on $C_x$.
		We obtain a cubic graph, in which the cycle corresponding to $C_x$ in $H_r$ is a triangle.
		Now, let $\set{j,k} = \set{1,2,3}\setminus\set{i}$.
		Subdivide the edges $x_ix_j$ and $x_ix_k$ twice each, 
		and label the new vertices as depicted in Figure~\ref{fig:case2-3incoming}.
		Finally, add the edges $y_{ij}y_k$ and $y_{ik}y_j$.
		\begin{figure}[htp!]
			$$
				\includegraphics{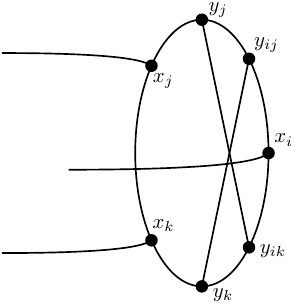}
			$$
			\caption{The edges in the cycle corresponding to $C_x^i$ in $H_r^i$ from Construction A.}
			\label{fig:case2-3incoming}
		\end{figure}
		
		It is straightforward to verify that $H_r^i$ satisfies the assumptions of the claim.
		Let $C_x^i$ denote the cycle in $H_r^i$ corresponding to $C_x$ in $H_r$,
		and let $F_r^i$ be the subgraph of $H_r^i$ corresponding to $F$, where $C_x$ is replaced by $C_x^i$.
		
		By the minimality of $H$ and since $H_r^i$ is not isomorphic to the Petersen graph, 
		it follows that $H_r^i$ contains at least one pair of 
		good $(C_r,C_x^i)$-paths with respect to $F_r^i$, each of which contains the edge $e_i$,
		since $e_j$ and $e_k$ cannot appear in a good pair of paths due to the adjacency of $x_j$ and $x_k$.
		
		\medskip
		\noindent {\bf Construction B} (for $|T_r| = k \ge 5$):

		Label the vertices of $T_r$ consecutively as they appear on $C_x$ by $x_1,\dots,x_k$,
		and denote the edge with $x_i$ as its only endvertex on $C_x$ by $e_i$, for every $i \in \set{1,\dots,k}$.
		
		Let $H_r^{0}$ be the graph obtained from $H_r$ by first removing all chords of $C_x$ (if any exist),
		and then suppressing all $2$-vertices on $C_x$.
		We obtain a cubic graph, in which the cycle corresponding to $C_x$ in $H_r$ is a $k$-cycle.

		Again, $H_r^0$ satisfies the assumptions of the claim.
		Let $C_x^0$ be the cycle in $H_r^0$ corresponding to $C_x$ in $H_r$,
		and let $F_r^0$ be the subgraph of $H_r^0$ corresponding to $F$, where $C_x$ is replaced by $C_x^0$.
		
		By the minimality of $H$, 
		we have that $H_r^0$ contains at least one pair of good $(C_r,C_x^0)$-paths with respect to $F_r^0$ or, 
		in the case when $|T_r|=5$, $H_r^0$ may be isomorphic to the Petersen graph.
		
		\bigskip
		Now, we are ready to proceed with the proof of the claim. 
		Define $t_1 = |T_1|$ and $t_2 = |T_2|$.
		We continue by considering the cases regarding the values of $t_1$ and $t_2$.
		
		Denote the vertices of $T_1$ and $T_2$ appearing consecutively on $C_x$ 
		in the counterclockwise order, by $x_1,\dots,x_{t_1}$ and $y_1,\dots,y_{t_2}$, respectively.		
		Moreover, we assume the labeling in such a way that in the counterclockwise order, 
		starting from $x_1$, $y_1$ appears as the first vertex from $T_2$.
		We will show that using the constructions defined above, we can find good pairs 
		of $(C_1,C_x)$-paths and $(C_2,C_x)$-paths.
		The key property that needs to be satisfied is that we can combine two pairs of good paths 
		in such a way that no two vertices on the final pair of paths are adjacent with respect to $F$.

		\medskip
		\noindent {\bf Case 2.1:} \quad	
		{\em $t_1 = 3$ and $t_2 = 3$.}

		\medskip
		\noindent {\bf Case 2.1.1:} \quad	
		{\em On $C_x$, between every pair of vertices from $T_1$ there is a vertex from $T_2$.}
		
		If $C_x$ contains no chord, then $C_x$ is a $6$-cycle and $C_x = x_1y_1x_2y_2x_3y_3$.
		The graph $H_1^1$, obtained by using Construction A, guarantees a good 
		pair of $(C_1,C_x)$-paths $P_1$ and $P_2$, 
		where one of them, say $P_1$, ends in $x_1$ 
		and $P_2$ ends either in $x_2$ or $x_3$, say $x_2$.
		
		Now, if one of the graphs $H_2^2$ and $H_2^3$ obtained from Construction A
		guarantees a good pair of $(C_2,C_x)$-paths $P_1'$ and $P_2'$ ending in the vertices $y_2$ and $y_3$, respectively,
		then a good pair of $(C_1,C_2)$-paths is formed by 
		$P_1 \cup x_1y_3 \cup P_2'$ and $P_2 \cup x_2y_2 \cup P_1'$.
		So, we may assume that no pair of good $(C_2,C_x)$-paths guaranteed by $H_2^2$ and $H_2^3$ ends in $y_2$ and $y_3$.
		Thus a good pair of $(C_2,C_x)$-paths $P_1'$ and $P_2'$ guaranteed by $H_2^2$ ends in $y_2$ and $y_1$, respectively.
		Similarly, a good pair of $(C_2,C_x)$-paths $P_3'$ and $P_4'$ guaranteed by $H_2^3$ ends in $y_3$ and $y_1$, respectively.
		But then, a good pair of $(C_1,C_x)$-paths $P_3$ and $P_4$ guaranteed by $H_1^3$ can be used.
		Namely, if $P_3$ and $P_4$ end in $x_3$ and $x_1$, respectively,
		then a good pair of $(C_1,C_2)$-paths is formed by 
		$P_3 \cup x_3y_2 \cup P_1'$ and $P_4 \cup x_1y_1 \cup P_2'$.
		Otherwise $P_3$ and $P_4$ end in $x_3$ and $x_2$, respectively;
		in this case, we use $P_3 \cup x_3y_3 \cup P_3'$ and $P_4 \cup x_2y_1 \cup P_4'$.
		
		Note that we can proceed in the same manner also when $C_x$ contains some chords,
		by some additional traversals of the edges of $C_x$ when combining the paths.

		\medskip
		\noindent {\bf Case 2.1.2:} \quad
		{\em There is exactly one pair of vertices from $T_1$ such that there are vertices from $T_2$ between them on $C_x$.}
		
		Without loss of generality, we assume that $x_1$ and $x_3$ are such a pair.
		By Claim~\ref{cl:2cut}, there is a chord in $C_x$ with an endvertex $d_1$ between at least one 
		of the two remaining pairs of vertices from $T_1$;
		without loss of generality, we may assume that $d_1$ is between $x_1$ and $x_2$.
		Now, we use a good pair of $(C_1,C_x)$-paths $P_1$ and $P_2$ guaranteed by $H_1^1$
		ending at $x_1$ and $x_i$, with $i \in \set{2,3}$, respectively.
		
		If there is an endvertex of some chord between $y_2$ and $y_3$, 
		then we use a good pair of $(C_2,C_x)$-paths $P_1'$ and $P_2'$ guaranteed by $H_2^3$
		ending at $y_3$ and $y_i$, with $i \in \set{1,2}$, respectively.
		Combining the paths $P_1$, $P_1'$ and $P_2$, $P_2'$ gives
		a good pair of $(C_1,C_2)$-paths.
		
		Otherwise, by Claim~\ref{cl:2cut}, 
		there must be an endvertex of some chord between $y_1$ and $y_2$,
		and we use a good pair of $(C_2,C_x)$-paths $P_3'$ and $P_4'$ guaranteed by $H_2^1$,
		where $P_3'$ ends in $y_1$.
		In this case, we combine the paths $P_1$, $P_4'$ and $P_2$, $P_3'$.
		
		\medskip
		\noindent {\bf Case 2.1.3:} \quad	
		{\em There is exactly one pair of vertices from $T_1$ such that there is no vertex from $T_2$ between them on $C_x$.}
		
		Without loss of generality, we may assume that $x_1$ and $x_2$ are such a pair.
		Therefore, $y_1$ is between $x_2$ and $x_3$, 
		and $y_3$ is between $x_1$ and $x_3$.		
		By the symmetry, we may also assume that $y_2$ is between $x_1$ and $x_3$.
		We consider two subcases.
		
		\medskip
		\noindent {\bf Case 2.1.3.1:} \quad
		{\em The vertices $x_1$ and $x_2$ are not adjacent.}
		
		Suppose first that $y_1$ is adjacent to $x_2$ and $x_3$.
		We use a good pair of $(C_1,C_x)$-paths $P_1$ and $P_2$ guaranteed by $H_1^1$,
		with $P_1$ ending in $x_1$ and $P_2$ ending in one of $x_i$, with $i \in \set{2,3}$.		
		If there is a good pair of $(C_2,C_x)$-paths $P_1'$ and $P_2'$ ending in $y_1$ and $y_3$, respectively,
		then we combine paths $P_1$, $P_2'$ and $P_2$, $P_1'$.
		
		Otherwise, in the case when $y_2$ is not adjacent to $x_3$,
		we use a good pair of $(C_2,C_x)$-paths $P_3'$ and $P_4'$ guaranteed by $H_2^1$ and ending in $y_1$ and $y_2$, respectively.
		We combine $P_1$, $P_4'$ and $P_2$, $P_3'$.
		So, we may assume that $y_2$ is adjacent to $x_3$.
		But then, by Claim~\ref{cl:special3cut}, $y_2$ and $y_3$ are not adjacent,
		and we use a good pair of $(C_2,C_x)$-paths $P_5'$ and $P_6'$ guaranteed by $H_2^3$ and ending in $y_3$ and $y_2$, respectively.
		In this case, we combine $P_1$, $P_5'$ and $P_2$, $P_6'$.
		
		Second, suppose that $y_1$ is not adjacent to $x_3$.
		We use a good pair of $(C_1,C_x)$-paths $P_1$ and $P_2$ guaranteed by $H_1^2$,
		with $P_1$ ending in $x_2$ and $P_2$ ending in one of $x_i$, with $i \in \set{1,3}$,
		and
		we use a good pair of $(C_2,C_x)$-paths $P_1'$ and $P_2'$ guaranteed by $H_2^1$,
		with $P_1'$ ending in $y_1$ and $P_2'$ ending in one of $y_i$, with $i \in \set{2,3}$.
		We combine $P_1$, $P_1'$ and $P_2$, $P_2'$.
		
		Finally, suppose that $y_1$ is not adjacent to $x_2$ (and is adjacent to $x_3$, due to the previous case).
		If $x_3$ is not adjacent to $y_2$, then
		we use a good pair of $(C_1,C_x)$-paths $P_1$ and $P_2$ guaranteed by $H_1^3$,
		with $P_1$ ending in $x_3$,
		and a good pair of $(C_2,C_x)$-paths $P_1'$ and $P_2'$ guaranteed by $H_2^1$,
		with $P_1'$ ending in $y_1$.
		We combine $P_1$, $P_1'$ and $P_2$, $P_2'$.
		Otherwise, $x_3$ is adjacent to $y_2$,
		and it follows by Claim~\ref{cl:special3cut} that $y_3$ is not adjacent to $y_2$.
		Therefore, we may use a good pair of $(C_1,C_x)$-paths $P_3$ and $P_4$ guaranteed by $H_1^1$,
		with $P_3$ ending in $x_1$,
		and a good pair of $(C_2,C_x)$-paths $P_3'$ and $P_4'$ guaranteed by $H_2^3$,
		with $P_3'$ ending in $y_3$.
		We combine $P_3$, $P_3'$ and $P_4$, $P_4'$.

		\medskip
		\noindent {\bf Case 2.1.3.2:} \quad
		{\em The vertices $x_1$ and $x_2$ are adjacent.}
		
		By Claim~\ref{cl:special3cut}, we have that $y_1$ is not adjacent to both $x_2$ and $x_3$.
		
		Suppose first that $y_1$ is not adjacent to $x_3$.
		We use a good pair of $(C_1,C_x)$-paths $P_1$ and $P_2$ guaranteed by $H_1^3$,
		with $P_1$ ending in $x_3$.
		If $P_2$ ends in $x_2$,
		then we use a good pair of $(C_2,C_x)$-paths $P_1'$ and $P_2'$ guaranteed by $H_2^1$,
		with $P_1'$ ending in $y_1$,
		and combine $P_1$, $P_2'$ and $P_2$, $P_1'$.
		Otherwise, $P_2$  ends in $x_1$.
		In this case, if $P_2'$ ends in $y_3$, then we combine $P_1$, $P_1'$ and $P_2$, $P_2'$,
		and if $P_2'$ ends in $y_2$, then we combine $P_1$, $P_2'$ and $P_2$, $P_1'$.
		
		Second, suppose that $y_1$ is not adjacent to $x_2$ (and is adjacent to $x_3$ by the previous case).
		We again use a good pair of $(C_1,C_x)$-paths $P_1$ and $P_2$ guaranteed by $H_1^3$,
		with $P_1$ ending in $x_3$,
		and a good pair of $(C_2,C_x)$-paths $P_1'$ and $P_2'$ guaranteed by $H_2^1$,
		with $P_1'$ ending in $y_1$.
		If $P_2'$ ends in $y_3$ or $y_2$ is not adjacent to $x_3$,
		then we combine $P_1$, $P_1'$ and $P_2$, $P_2'$.
		Otherwise, $P_2$ ends in $y_2$ and $y_2$ is adjacent to $x_3$.
		Then, by Claim~\ref{cl:special3cut}, $y_2$ is not adjacent to $y_3$.
		We take a good pair of $(C_2,C_x)$-paths $P_3'$ and $P_4'$ guaranteed by $H_2^3$,
		with $P_3'$ ending in $y_3$,
		and combine $P_2$, $P_3'$ and $P_1$, $P_4'$.

		\medskip
		\noindent {\bf Case 2.2:} \quad	
		{\em $t_1 = 3$ and $t_2 \ge 5$.}
		Note that by the symmetry, this case also covers the case $t_1 \ge 5$ and $t_2 = 3$.		
		
		Suppose first that $t_2 = 5$ and $H_2$ is isomorphic to the Petersen graph.
		Label the five vertices on $C_2$ by $z_1,z_2,\dots,z_5$ in the counterclockwise order,
		where $z_1$ is adjacent to $y_1$.
		\begin{figure}[htp!]
			$$
				\includegraphics{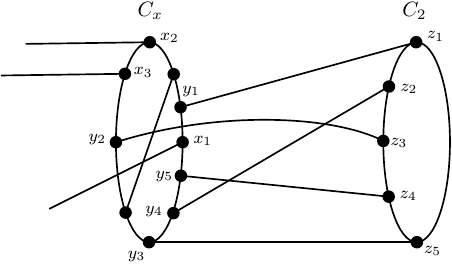}
			$$
			\caption{An example of the vertices on the cycles $C_x$ and $C_2$ when $|T_1|=3$ and $H_2$ is isomorphic to the Petersen graph.}
			\label{fig:case2-H2Pet}
		\end{figure}
		
		By Claim~\ref{cl:2cut}, no vertex from $T_1$ is adjacent to both of the other two vertices from $T_1$;
		therefore, since $t_1 = 3$, at least one vertex from $T_1$ is not adjacent to any of the others.		
		Thus, without loss of generality, we may assume that $x_1$ is adjacent to neither $x_2$ nor $x_3$
		(see Figure~\ref{fig:case2-H2Pet} for the labeling of the vertices).
		
		Let $P_1$ and $P_2$ be a good pair of $(C_1,C_x)$-paths guaranteed by $H_1^1$. 
		We may assume, due to symmetry, that they end in $x_1$ and $x_2$, respectively.
		Now, suppose that all vertices of $T_2$ appear between two vertices of $T_1$.
		Then, we combine $P_1$, $y_1z_1$ and $P_2$, $y_5z_4$ or $P_1$, $y_5z_4$ and $P_2$, $y_1z_1$ (depending on the adjacencies).
		Hence, we may assume that not all vertices of $T_2$ lie between two vertices of $T_1$.
		
		Suppose that $y_4$ is not adjacent to $y_5$. 
		If $y_4$ and $y_5$ are not both adjacent to $x_2$,
		then we can combine $P_1$, $y_5z_4$ and $P_2$, $y_4z_2$
		or $P_1$, $y_4z_2$ and $P_2$, $y_5z_4$.
		If $y_4$ and $y_5$ are both adjacent to $x_2$, 
		then we combine $P_1$, $y_1z_1$ and $P_2$, $y_5z_4$.
		
		So, we may assume that $y_4$ is adjacent to $y_5$. 
		If $y_1$ is not adjacent to $x_1$,
		then we combine $P_1$, $y_5z_4$ and $P_2$, $y_1z_1$.
		Hence, we may also assume that $y_1$ and $x_1$ are adjacent.
		Similarly, if $y_1$ is not adjacent to $x_2$,
		then we combine $P_1$, $y_1z_1$ and $P_2$, $y_2z_3$.
		Thus, also $y_1$ and $x_2$ are adjacent.
		
		By Claim~\ref{cl:special3cut}, $x_2$ is not adjacent to $x_3$.
		In this case, let $P_3$ and $P_4$ be a good pair of $(C_1,C_x)$-paths guaranteed by $H_1^3$,
		with $P_3$ ending in $x_3$. 
		If $P_4$ ends in $x_2$, 
		then we combine $P_3$, $y_5z_4$ and $P_4$, $y_1z_1$.
		Otherwise, $P_4$ ends in $x_1$, 
		and we combine $P_3$, $y_2z_3$ and $P_4$, $y_1z_1$.

						
		\bigskip				
		Thus, we may assume that $H_2$ is not isomorphic to the Petersen graph. 
		Then, there exists a good pair of $(C_2,C_x)$-paths $Q_1$ and $Q_2$, guaranteed by $H_2^0$,
		ending in $y_i$ and $y_{j}$, respectively,
		for some $i,j \in \set{1,\dots,t_2}$ with $2 \le |i-j| \le t_2 - 2$. 
		We can also assume, similarly to the previous case, that $x_1$ is adjacent to neither $x_2$ nor $x_3$.
		Hence, we may further assume that there exists a good pair of $(C_1,C_x)$-paths $P_1$ and $P_2$ guaranteed by $H_1^1$,
		ending in $x_1$ and $x_t$, where $t \in \set{2,3}$, respectively. Let $x_s$ be the vertex distinct from $x_1$ and $x_t$ (i.e., $s = 5 - t$).
		Additionally, let $P_3$ and $P_4$ be a good pair of $(C_1,C_x)$-paths guaranteed by $H_1^s$,
		ending in $x_s$ and $x_1$ or $x_t$.
				
		If the two endvertices of $P_1$ and $P_2$ lie on the same $[y_i,y_j]_C$-segment $S$
		(without loss of generality, assume that $x_1$ is between $y_i$ and $x_t$),
		then we combine $P_1$, $Q_1$ and $P_2$, $Q_2$.
		
		Otherwise, $x_1$ lies on one of the $[y_i,y_j]_C$-segments while $x_t$ lies on the other.		
		If $x_1$ is not adjacent to $y_j$ and $x_t$ is not adjacent to $y_i$,
		then we combine $P_1$, $Q_1$ and $P_2$, $Q_2$.
		
		So, suppose first that $x_1$ is adjacent to $y_j$. 
		Then, $x_1$ is not adjacent to $y_i$, since $2 \le |i-j| \le t_2 - 2$.
		If $x_t$ is not adjacent to $y_j$, 
		then we combine $P_1$, $Q_2$ and $P_2$, $Q_1$.
		Otherwise $x_t$ is adjacent to $y_j$ and, by Claim~\ref{cl:special3cut}, $x_t$ and $x_s$ are not adjacent.
		In this case, we combine $P_3$, $Q_1$ and $P_4$, $Q_2$.
		
		Finally, suppose that $x_t$ is adjacent to $y_i$ and $x_1$ is not adjacent to $y_j$ (recall that $x_t$ is not adjacent to $y_j$, since $2 \le |i-j| \le t_2 - 2$).
		If $x_1$ is not adjacent to $y_i$, then we combine $P_1$, $Q_2$ and $P_2$, $Q_1$.
		So, we may assume that $x_1$ is adjacent to $y_i$, and by Claim~\ref{cl:special3cut}, $x_t$ and $x_s$ are not adjacent.
		In this case, we combine $P_3$, $Q_2$ and $P_4$, $Q_1$.
		
		This completes the proof of Case~2.2.
		
		\medskip
		\noindent {\bf Case 2.3:} \quad	
		{\em $t_1 \ge 5$ and $t_2 \ge 5$.}
		
		Suppose first that both $H_1$ and $H_2$ are isomorphic to the Petersen graph.
		We follow the labeling from Case~2.2, and so we label the vertices of $H$ as depicted in Figure~\ref{fig:case2-H1-H2Pet}.
		\begin{figure}[htp!]
			$$
				\includegraphics{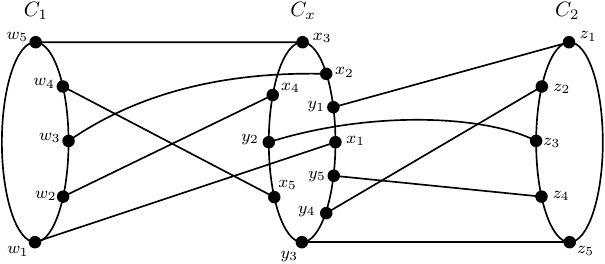}
			$$
			\caption{An example of the vertex structure on the cycles $C_x$, $C_1$, and $C_2$ 
				when $H_1$ and $H_2$ are both isomorphic to the Petersen graph.}
			\label{fig:case2-H1-H2Pet}
		\end{figure}
		Without loss of generality, assume that between $x_1$ and $y_5$ there is no other vertex from $T_1$ or $T_2$.
		
		Suppose that $x_2$ is adjacent to $x_1$.
		In this case, if $x_5$ is not adjacent to $y_5$, 
		then we form a good pair of $(C_1,C_2)$-paths by combining $w_1x_1$, $y_5z_4$ and $w_4x_5$, $y_1z_1$.
		So, we may assume that $x_5$ is adjacent to $y_5$.
		Further, if $x_4$ is not adjacent to $x_5$,
		then we combine $w_4x_5$, $y_5z_4$ and $w_2x_4$, $y_1z_1$.
		Hence, we also have that $x_4$ is adjacent to $x_5$.
		Finally, if $x_3$ is not adjacent to $x_4$, 
		then we combine $w_2x_4$, $y_5z_4$ and $w_5x_3$, $y_1z_1$.
		It follows that $x_3$ is adjacent to $x_4$.
		But this means that $x_2$ and $x_3$ are not adjacent, 
		and we combine $w_3x_2$, $y_5z_4$ and $w_5x_3$, $y_4z_2$.
		
		Therefore, we may assume that $x_2$ is not adjacent to $x_1$.
		If $y_4$ is not adjacent to $y_5$, 
		then we combine $w_1x_1$, $y_5z_4$ and $w_3x_2$, $y_4z_2$.
		So, $y_4$ is adjacent to $y_5$.
		If $y_3$ is not adjacent to $y_4$, 
		then we combine $w_1x_1$, $y_4z_2$ and $w_3x_2$, $y_3z_5$.
		So, $y_3$ is adjacent to $y_4$.
		Similarly, we deduce that $y_2$ is adjacent to $y_3$,
		and that $y_1$ is adjacent to $y_2$,
		a contradiction to Claim~\ref{cl:2cut}.
		
		\bigskip
		Hence, at least one of $H_1$ and $H_2$ is not isomorphic to the Petersen graph.
		
		First, suppose that $H_1$ or $H_2$, say $H_1$, is isomorphic to the Petersen graph.
		Let $Q_1$, $Q_2$ be a good pair of $(C_2,C_x)$-paths guaranteed by $H_2^0$,
		ending in $y_i$ and $y_j$, $2 \le |i-j| \le t_2-2$, respectively.
		
		If there is $k \in \set{1,\dots,5}$ such that $x_k$ is not adjacent to $x_{k+1}$ (indices taken modulo $5$) and 
		$x_k$, $x_{k+1}$ are both in the same $[y_i,y_j]_{C_x}$-segment, 
		then, since $H_1$ is isomorphic to the Petersen graph, there is a good pair of $(C_1,C_2)$-paths.
		
		Therefore, we may assume that there is a pair $x_k$, $x_{k+1}$ lying on distinct $[y_i,y_j]_{C_x}$-segments.
		If $x_k$ and $x_{k+1}$ are both adjacent neither to $y_i$ nor $y_j$, 
		then it is again easy to find a good pair of $(C_1,C_2)$-paths
		(since there is at least one other vertex from $T_2$ on each of the two segments).
		So, $x_k$ and $x_{k+1}$ are, without loss of generality, both adjacent to $y_i$.
		If $x_{k-1}$ is not on the same $[y_i,y_j]_{C_x}$-segment as $x_k$ (or $x_{k+2}$ is not on the same segment as $x_{k+1}$), 
		then, by the same argument as above, we can again find a good pair of paths.
		Thus, on one $[y_i,y_j]_{C_x}$-segment, there is only a pair of vertices, say $x_k$ and $x_{k-1}$.
		We already know that $x_k$ and $x_{k-1}$ are not adjacent (due to Claim~\ref{cl:special3cut}), 
		so we can find a good pair of $(C_1,C_2)$-paths.
		
		Finally, we may assume that none of $H_1$ and $H_2$ is isomorphic to the Petersen graph.
		Let $P_1$, $P_2$ be a good pair of $(C_1,C_x)$-paths guaranteed by $H_1^0$,
		ending in $x_k$ and $y_\ell$, $2 \le |k-\ell| \le t_1-2$, respectively,
		and
		let $Q_1$, $Q_2$ be a good pair of $(C_2,C_x)$-paths guaranteed by $H_2^0$,
		ending in $y_i$ and $y_j$, $2 \le |i-j| \le t_2-2$, respectively.
		If $x_k$ and $x_{\ell}$ lie on the same $[y_i,y_j]_{C_x}$-segment (assuming that $x_k$ is closer to $y_i$),
		then we combine the paths $P_1$, $Q_1$ and $P_2$, $Q_2$.
		Otherwise, we consider three cases.
		If $x_k$ is not adjacent to $y_j$ and $x_\ell$ is not adjacent to $y_i$,
		then we combine the paths $P_1$, $Q_1$ and $P_2$, $Q_2$.
		If $x_k$ is adjacent to $y_j$, then $x_k$ is not adjacent to $y_i$ and $x_\ell$ is not adjacent to $y_j$,
		and we combine the paths $P_1$, $Q_2$ and $P_2$, $Q_1$.
		And finally, if $x_k$ is not adjacent to $y_j$ and $x_\ell$ is adjacent to $y_i$,
		then $x_k$ is not adjacent to $y_i$ and $x_\ell$ is not adjacent to $y_j$,
		and we can again combine the paths $P_1$, $Q_2$ and $P_2$, $Q_1$.
		
		This completes the proof of the claim.
	\end{proofclaim}

	\bigskip
	Now, we construct a normal $6$-edge-coloring of $G$.
	By Claim~\ref{cl:goodpaths}, there exists a pair of good $(C_1,C_2)$-paths $P_1 = [x_1,x_2]$ and $P_2 = [y_1,y_2]$.
	Note that every even cycle $C$ of $F$ is incident with either zero or two vertices of $P_1$,
	which are incident with one edge of $P_1$ on $C$.
	In the latter case, denote the vertex closer to $C_1$ by $x_{C,1}$ and the other by $x_{C,2}$.
	We refer to $x_{C,1}$ as {\em incoming for $P_1$} and to $x_{C,2}$ as {\em outgoing for $P_1$}.
	Similarly, we label the potential two vertices incident with $P_2$ by $y_{C,1}$ and $y_{C,2}$, 
	and refer to them as {\em incoming for $P_2$} and {\em outgoing for $P_2$}, respectively.
	
	We color every edge of $M$ that does not lie on $P_1$ or $P_2$ with color $1$,
	and the edges of $M$ lying on $P_1$ or $P_2$ with color $2$.
	
	Next, we color $C_1$ and $C_2$.
	Let $S_1$ and $S_1'$ be the two $[x_1,y_1]_{C_1}$-segments of $C_1$.
	We color the edges of $S_1$ alternately with colors $3$ and $4$,
	and the edges of $S_1'$ alternately with colors $5$ and $6$,
	such that the two edges incident with $x_1$ are colored with $3$ and $5$,
	and the two edges incident with $y_1$ are colored with either $3$ and $6$, or $4$ and $5$.
	In the same manner, we color the two $[x_2,y_2]_{C_2}$-segments of $C_2$,
	such that the two edges incident with $x_2$ are colored with $3$ and $5$,
	and the two edges incident with $y_2$ are colored with either $3$ and $6$, or $4$ and $5$.
	
	Finally, let $C$ be an even cycle of $F$.	
	If $C$ is incident neither with edges of $P_1$ nor $P_2$,
	then we color its edges alternately with colors $3$ and $4$.
	If $C$ is incident only with vertices of $P_1$ (or $P_2$),
	then we color one of the two segments $[x_{C,1},x_{C,2}]_C$ alternately with colors $3$ and $4$,
	and the other alternately with colors $5$ and $6$ such that $x_{C,1}$ 
	is incident with edges of colors $3$ and $5$ (or, respectively, $y_{C,1}$ is incident with edges of colors $3$ and $6$).
	
	If $C$ is incident with vertices of $P_1$ as well as of $P_2$, 
	we first, without loss of generality, assume that the vertices $x_{C,1}$, $x_{C,2}$, $y_{C,2}$, $y_{C,1}$
	appear around $C$ in the given order.	
	Now, we color the segment $[x_{C,1},x_{C,2}]_C$, that does not contain any vertex of $P_2$,
	alternately with colors $3$ and $4$;	
	next we color the segment $[x_{C,2},y_{C,2}]_C$, that does not contain $x_{C,1}$,
	alternately with colors $5$ and $6$ such that $x_{C,2}$ is incident with edges of colors either $3$ and $5$, or $4$ and $6$.
	We proceed by coloring the segment $[y_{C,1},y_{C,2}]_C$, that does not contain $x_{C,1}$,
	alternately with colors $3$ and $4$ such that $y_{C,2}$ is incident with edges of colors either $3$ and $6$, or $4$ and $5$.
	We finish by coloring the segment $[x_{C,1},y_{C,1}]_C$, that does not contain $x_{C,2}$,
	alternately with colors $5$ and $6$ such that $y_{C,1}$ is incident with edges of colors either $3$ and $6$, or $4$ and $5$ (see Figure~\ref{fig:final} for an example).
	Since $C$ is even, it follows that $x_{C,1}$ is incident with edges of colors either $3$ and $5$, or $4$ and $6$.
	\begin{figure}[htp!]
		$$
			\includegraphics{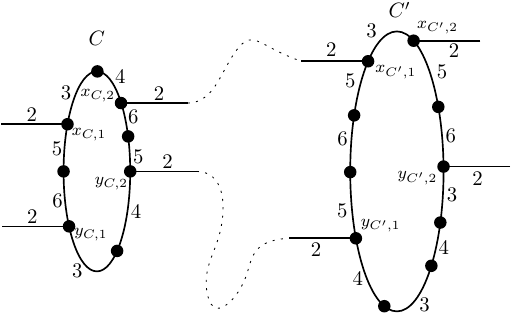}
		$$
		\caption{An example of a coloring of two even cycles.}
		\label{fig:final}
	\end{figure}	
	
	It remains to show that the constructed coloring is indeed normal.
	We first consider the edges of $M$.
	If an edge of $M$ is also an edge of $P_1$, 
	then each of its end-vertices is incident either with colors $3$ and $5$, or $4$ and $6$, hence it is normal.
	Similarly an edge of $P_2$ is incident either with colors $3$ and $6$, or $4$ and $5$ and therefore normal.
	If an edge of $M$ is neither an edge of $P_1$ nor $P_2$,
	then, at each endvertex, it is incident either with colors $3$ and $4$, or $5$ and $6$ and therefore normal.
	
	Next, we consider the edges of $F$.
	Suppose first that an edge $e$ is incident with an incoming and an outgoing vertex of $P_1$ (or $P_2$).
	Then $e$ is incident with two edges of color $2$ and two edges of the same color $c$ from $\set{3,4,5,6}$, hence it is poor.
	If $e$ is incident only with an incoming or an outgoing vertex of $P_1$ (or $P_2$),
	then it is incident with four distinct colors, and thus rich.
	If $e$ is incident neither with an incoming nor an outgoing vertex, then it is poor.
	
	This means that all the edges of $G$ are normal and hence the proof is completed.
\end{proof}

\section{Conclusion}

In this paper, we extended the approach of Mazzuoccolo and Mkrtchyan~\cite{MazMkr20b} 
for constructing a normal $6$-edge-coloring of permutation snarks to cubic graphs of oddness $2$.
It turned out that this extension required significantly more technical considerations, and consequently, 
we do not see a straightforward way to further generalize our proof to graphs of oddness $k$. 
Even the case with $k=4$ appears to be challenging.

On the other hand, recall that Conjecture~\ref{con:Nor} states that $5$ colors are sufficient to color any 
bridgeless cubic graph. 
A first step toward resolving this conjecture could thus be to find a technique for constructing
a normal $5$-edge-coloring of permutation snarks.

\paragraph{Acknowledgement.} 
I.~Fabrici, R.~Sot\'{a}k, and D.~\v{S}vecov\'{a} were supported by the Slovak Research and Development Agency under the contract APVV--23--0191.
B.~Lu\v{z}ar was partially supported by the Slovenian Research and Innovation Agency Program P1--0383 and the projects J1--3002 and J1--4008.
	
\bibliographystyle{plain}
{
	\bibliography{References}

\begin{thebibliography}{10}

\bibitem{FerMazMkr20}
L.~Ferrarini, G.~Mazzuoccolo, and V.~Mkrtchyan.
\newblock {Normal 5-edge-colorings of a family of Loupekhine snarks}.
\newblock {\em AKCE Int. J. Graphs Comb.}, 17:720--724, 2020.

\bibitem{HagSte13}
J.~H\"{a}gglund and E.~Steffen.
\newblock {Petersen-colorings and some families of snarks}.
\newblock {\em Ars Math. Contemp.}, 7:161--173, 2014.

\bibitem{Jae85}
F.~Jaeger.
\newblock On five-edge-colorings of cubic graphs and nowhere-zero flow
  problems.
\newblock {\em Ars Combin.}, 20B:229--244, 1985.

\bibitem{Jae88}
F.~Jaeger.
\newblock Nowhere-zero flow problems.
\newblock In {\em Selected topics in graph theory}, pages 71--95. Academic
  Press, 1988.

\bibitem{LukMacMazSko15}
R.~Lukotka, E.~M\'{a}\v{c}ajov\'{a}, J.~Maz\'{a}k, and M.~\v{S}koviera.
\newblock {Small Snarks with Large Oddness}.
\newblock {\em Electron. J. Combin.}, 22, 2 2015.

\bibitem{MazMkr20b}
G.~Mazzuoccolo and V.~Mkrtchyan.
\newblock {Normal 6-edge-colorings of some bridgeless cubic graphs}.
\newblock {\em Discrete Appl. Math.}, 277:252--262, 2020.

\bibitem{MazMkr20}
G.~Mazzuoccolo and V.~Mkrtchyan.
\newblock {Normal edge-colorings of cubic graphs}.
\newblock {\em J. Graph Theory}, 94:75--91, 5 2020.

\bibitem{Pet1891}
J.~Petersen.
\newblock Die {T}heorie der regul\"{a}ren {G}raphs.
\newblock {\em Acta Math.}, 15:193--220, 1891.

\bibitem{SedSkr24b}
J.~Sedlar and R.~\v{S}krekovski.
\newblock {Normal 5-edge-coloring of some snarks superpositioned by Flower
  snarks}.
\newblock {\em European J. Combin.}, 122:104038, 2024.

\bibitem{SedSkr24}
J.~Sedlar and R.~\v{S}krekovski.
\newblock {Normal 5-edge-coloring of some snarks superpositioned by the
  Petersen graph}.
\newblock {\em Appl. Math. Comput.}, 467:128493, 2024.

\bibitem{ZhoHaoLuoLuo26}
W.~Zhou, R.~Hao, R.~Luo, and Y.~Luo.
\newblock An infinite family of normal 5-edge colorable superpositioned snarks.
\newblock {\em Discrete Appl. Math.}, 378:259--269, 2026.

\end{thebibliography}
}

\end{document}